\documentclass{amsart}
\usepackage{amssymb}
\usepackage{enumitem}
\usepackage[colorlinks = true,
linkcolor = black,
urlcolor = blue,
citecolor = black]{hyperref}
\usepackage[margin=1in]{geometry}
\allowdisplaybreaks

\theoremstyle{definition} 
\newtheorem{thm}{Theorem}[section]
\newtheorem{cor}[thm]{Corollary}

\newtheorem{lem}[thm]{Lemma}

\newtheorem{rmk}[thm]{Remark}
\newtheorem{exmp}[thm]{Example}
\newtheorem{defn}[thm]{Definition}

\newcommand{\bb}{\mathbb}

\newcommand{\bZ}{\bb{Z}}
\newcommand{\bQ}{\bb{Q}}

\newcommand{\bF}{\bb{F}}
\newcommand{\bE}{\bb{E}}
\newcommand{\bP}{\bb{P}}
\newcommand{\fp}{\bF_p}
\newcommand{\zp}{\bZ_p}

\newcommand{\lt}{\left |}
\newcommand{\rt}{\right |}
\newcommand{\lf}{\left \lfloor}
\newcommand{\rf}{\right \rfloor}
\newcommand{\mr}{\mathrm}
\newcommand{\mf}{\mathfrak}
\newcommand{\al}{\alpha}

\DeclareMathOperator{\cok}{\mr{cok}}
\DeclareMathOperator{\Sur}{\mr{Sur}}
\DeclareMathOperator{\Hom}{\mr{Hom}}
\DeclareMathOperator{\Aut}{\mr{Aut}}
\DeclareMathOperator{\M}{\mr{M}}
\DeclareMathOperator{\rank}{\mr{rank}}
\begin{document}

\title[Inhomogeneously balanced random $p$-adic matrices]{Universality of the cokernels of random $p$-adic matrices with inhomogeneously balanced columns}
\author{Jungin Lee and Sungjin Park}
\date{}
\address{J. Lee -- Department of Mathematics, Ajou University, Suwon 16499, Republic of Korea \newline
\indent S. Park -- Department of Mathematics, Ajou University, Suwon 16499, Republic of Korea}
\email{jileemath@ajou.ac.kr, qkrtjd5421@ajou.ac.kr}

\begin{abstract}
In this paper, we prove universality of the distribution of the cokernels of random $p$-adic matrices with inhomogeneously balanced columns. More precisely, let $u \ge 0$ be an integer, and for each positive integer $n$, let $A(n)$ be a random $n \times (n+u)$ matrix over $\mathbb{Z}_p$ whose $i$-th column is $\alpha_n(i)$-balanced. We prove that if $\sum_{i=1}^{n+u} \exp(-\epsilon \alpha_n(i)n) \to 0$ as $n \to \infty$ for every $\epsilon>0$, then the cokernels of $A(n)$ converge in distribution, as $n \to \infty$, to the same limiting law as the cokernels of Haar-random $n \times (n+u)$ matrices over $\mathbb{Z}_p$. This extends a universality theorem of Nguyen and Wood to random $p$-adic matrices with inhomogeneously balanced columns.
\end{abstract}
\maketitle

\vspace{-5mm}
\section{Introduction} \label{Sec1}

Let $p$ be a prime. We denote by $\fp$ the finite field with $p$ elements and by $\zp$ the ring of $p$-adic integers. For a commutative ring $R$, we denote by $\M_{m \times n}(R)$ the set of all $m \times n$ matrices over $R$. The Cohen--Lenstra heuristics \cite{CL84} predict the distribution of the ideal class groups of imaginary quadratic number fields. 
Motivated by the function field analogue of the Cohen--Lenstra heuristics, Friedman and Washington \cite{FW89} studied the distribution of the cokernels of random $p$-adic matrices. They proved that if $A(n)$ is a random matrix in $\M_n(\zp)$ distributed according to the normalized Haar measure for each $n \ge 1$, then
\begin{equation*}
\lim_{n \to \infty} \bP(\cok(A(n)) \cong H) = \frac{1}{|\Aut(H)|} \prod_{i=1}^{\infty} (1-p^{-i})
\end{equation*}
for every finite abelian $p$-group $H$. Here, $\cok(A(n))$ denotes the cokernel of a matrix $A(n)$ and $\Aut(H)$ denotes the automorphism group of $H$.

Let $\al \in (0, 1/2]$ be a real number. Let $R$ be either $\bZ$, $\zp$, or $\bZ/a\bZ$ for some integer $a \ge 2$. A random variable $y$ taking values in $R$ is \textit{$\al$-balanced} if
\begin{equation*}
\bP(y \equiv r \pmod{\mf{m}}) \le 1-\al
\end{equation*}
for every maximal ideal $\mf{m}$ of $R$ and $r\in R/\mf{m}$. A random vector or matrix with entries in $R$ is \textit{$\al$-balanced} if its entries are independent and $\al$-balanced. We note that in the definition of an $\al$-balanced matrix, the entries are not required to be identically distributed.
Nguyen and Wood \cite[Theorem 4.1]{NW22} proved a universality theorem extending a result of Wood \cite{Woo19}. We recall the following $p$-adic form.

\begin{thm} \label{thm1a}
(\cite[Theorem 3.4]{Woo22}) Let $u \ge 0$ be an integer. Let $(\al_n)_{n \ge 1}$ be a sequence in $(0, 1/2]$ such that for every constant $\Delta > 0$, we have $\al_n \ge \frac{\Delta \log n}{n}$ for all sufficiently large $n$. Let $A(n)$ be an $\al_n$-balanced random matrix in $\M_{n \times (n+u)}(\zp)$ for each $n \ge 1$. Then for every finite abelian $p$-group $H$,
\begin{equation*}
\lim_{n \to \infty} \bP(\cok(A(n)) \cong H) = \frac{1}{|H|^u |\Aut(H)|} \prod_{i=1}^{\infty} (1-p^{-i-u}).
\end{equation*}
\end{thm}

By reduction modulo $p$, we obtain the following corollary for random matrices over $\fp$. We note that for an $n \times (n+u)$ matrix $A$ over $\fp$, $\rank_{\fp}(A)=n-k$ if and only if $\cok(A) \cong \fp^k$.

\begin{cor} \label{cor1b}
(\cite[Theorem 3.2]{Woo22}) Let $u \ge 0$ be an integer. Let $(\al_n)_{n \ge 1}$ be a sequence in $(0, 1/2]$ such that for every constant $\Delta > 0$, we have $\al_n \ge \frac{\Delta \log n}{n}$ for all sufficiently large $n$. Let $A(n)$ be an $\al_n$-balanced random matrix in $\M_{n \times (n+u)}(\fp)$ for each $n \ge 1$. Then for every nonnegative integer $k$,
\begin{equation*}
\lim_{n \to \infty} \bP(\rank_{\fp}(A(n)) = n-k)
= p^{-k(k+u)} \frac{\prod_{j=1}^{\infty} (1-p^{-j})}{\prod_{j=1}^{k} (1-p^{-j}) \prod_{j=1}^{k+u} (1-p^{-j})}.
\end{equation*}
\end{cor}

Similar universality results were proved for random symmetric \cite[Theorem 1.3]{Woo17} and alternating \cite[Theorem 1.13]{NW25} matrices over $\zp$, and for random Hermitian matrices over the ring of integers of a quadratic extension of $\bQ_p$ \cite[Theorem 1.6]{Lee23b}, in the case where $\al_n=\al>0$ is constant. Recently, in the case $u=0$, the first author \cite[Theorem 1.3]{Lee25} improved Corollary \ref{cor1b} to the sharp threshold $\al_n = \frac{c \log n}{n}$ for every constant $c>1$. 
Moreover, Jung, the first author, and Yu \cite[Theorem 1.5]{JLY26} established the corresponding sharp-threshold universality results for random non-symmetric, symmetric, and alternating matrices over $\zp$.

In all of the above results, all entries of the random matrix are assumed to satisfy the same balancedness condition. In this paper, we consider more general random $p$-adic matrices whose columns may satisfy different balancedness conditions. The following theorem is the main result of this paper.

\begin{thm} \label{thm1c}
Let $u \ge 0$ be an integer and $\al_n(1), \al_n(2), \ldots, \al_n(n+u) \in (0, 1/2]$ for every $n \ge 1$. Assume that
for every $\epsilon>0$,
\begin{equation} \label{eq1_condition}
\lim_{n \to \infty} \sum_{i=1}^{n+u} \exp \left( - \epsilon \al_n(i) n \right) = 0.
\end{equation}
Let $A(n)$ be a random matrix in $\M_{n \times (n+u)}(\zp)$ with independent entries such that the $i$-th column is $\al_n(i)$-balanced for every $i$. Then for every finite abelian $p$-group $H$,
\begin{equation*}
\lim_{n \to \infty} \bP(\cok(A(n)) \cong H) = \frac{1}{|H|^u |\Aut(H)|} \prod_{i=1}^{\infty} (1-p^{-i-u}).
\end{equation*}
\end{thm}

\begin{cor} \label{cor1d}
Let $u \ge 0$ be an integer and $\al_n(1), \al_n(2), \ldots, \al_n(n+u)$ be as in Theorem \ref{thm1c}. Let $A(n)$ be a random matrix in $\M_{n \times (n+u)}(\fp)$ with independent entries such that the $i$-th column is $\al_n(i)$-balanced for every $i$. Then for every nonnegative integer $k$,
\begin{equation*}
\lim_{n \to \infty} \bP(\rank_{\fp}(A(n)) = n-k)
= p^{-k(k+u)} \frac{\prod_{j=1}^{\infty} (1-p^{-j})}{\prod_{j=1}^{k} (1-p^{-j}) \prod_{j=1}^{k+u} (1-p^{-j})}.
\end{equation*}
\end{cor}

\begin{rmk} \label{rmk1e}
Assume that $\al_n(i) = \al_n$ for all $1 \le i \le n+u$. Then 
\begin{equation*}
\lim_{n \to \infty} \sum_{i=1}^{n+u} \exp \left( - \epsilon \al_n(i) n \right) = \lim_{n \to \infty} \exp \left( \log(n+u) - \epsilon \al_n n \right)=0
\end{equation*}
for every $\epsilon>0$ if and only if, for every $\Delta > 0$, we have $\al_n \ge \frac{\Delta \log n}{n}$ for all sufficiently large $n$. Hence Theorem \ref{thm1c} and Corollary \ref{cor1d} generalize Theorem \ref{thm1a} and Corollary \ref{cor1b}, respectively.
\end{rmk}

We determine the limiting distribution of $\cok(A(n))$ by computing its surjective moments. By Wood's moment theorem \cite[Theorem 3.1]{Woo19}, together with the argument of \cite[Corollary 3.4]{Woo19}, the following theorem implies Theorem \ref{thm1c}. For $\zp$-modules $M$ and $N$, we write $\Sur(M,N)$ for the set of surjective $\zp$-module homomorphisms from $M$ to $N$.

\begin{thm} \label{thm1f}
Let $u \ge 0$ be an integer and $A(n)$ be a random matrix in $\M_{n \times (n+u)}(\zp)$ defined as in Theorem \ref{thm1c}. Then for every finite abelian $p$-group $G$, 
\begin{equation*}
\lim_{n \to \infty} \bE(\#\Sur(\cok(A(n)), G)) = \frac{1}{|G|^u}.
\end{equation*}
\end{thm}

It would be desirable to extend the columnwise inhomogeneous random matrix model of our main theorem to an entrywise inhomogeneous model. More precisely, we may allow the $(r,s)$-entry to have its own balancedness parameter $\al_n(r,s)$. In such a model, a criterion for universality should depend not only on the multiset of these parameters, but also on their positions in the matrix. 

For example, if the first row is identically zero, then the cokernel is always infinite. On the other hand, \cite[Proposition 2.3]{KLY24} shows that when $u=0$ and the upper-left $a_n \times b_n$ block is identically zero, while all remaining entries are independent and Haar-random, then the cokernel converges to the Cohen--Lenstra distribution if and only if $\lim_{n \to \infty} (n-a_n-b_n) = \infty$. Taking $a_n=b_n=\lf \sqrt{2n} \rf$, we have $\lim_{n \to \infty} (n-a_n-b_n) = \infty$ and $a_n b_n \ge n$ for all $n \ge 1$. 
These examples demonstrate that the positions of the balancedness parameters play an essential role in any entrywise universality criterion.

We briefly outline the proof of Theorem \ref{thm1f}. The proof follows the moment computation in \cite[Section 4]{NW22}, but some modifications are needed in our setting. The main difference is that the columns of $A(n)$ are inhomogeneously balanced, so the moment estimates have to keep track of the varying balancedness parameters $\al_n(i)$. Lemma \ref{lem2c} handles the contribution of codes under the condition \eqref{eq1_condition}, while Lemma \ref{lem2h} provides the summation estimate needed to control the contribution of non-codes. These two lemmas are the main new ingredients in adapting the Nguyen--Wood argument to the present setting.

\section{Computation of the moments} \label{Sec2}

Throughout this section, let $a=p^d$ for a positive integer $d$, let $R=\bZ/a\bZ$ and fix $u\ge 0$. For each positive integer $n$, let $V=R^n$ and $W=R^{n+u}$. We write $v_1,\ldots,v_n$ and $w_1,\ldots,w_{n+u}$ for the standard bases of $V$ and $W$, respectively. If $\sigma\subset [n]$, let $V_{\backslash\sigma}$
denote the submodule of $V$ generated by the $v_i$ with $i \notin \sigma$.
For each positive integer $n$, let $\al_n(1), \ldots, \al_n(n+u) \in (0, 1/2]$. Assume that this numbers satisfy the condition \eqref{eq1_condition} for every $\epsilon>0$. Let $A(n)$ be a random matrix in $\M_{n\times(n+u)}(R)$ with independent entries such that the $i$-th column is $\al_n(i)$-balanced. Let $G$ be a finite abelian $p$-group such that $p^d G=0$. If we regard $A(n)$ as a linear map from $W$ to $V$, then 
\begin{equation*}
\begin{split}
\bE(\# \Sur(\cok(A(n)), G)) &=\sum_{F\in \Sur(V, G)} \bP(FA(n)=0) \\
&=\sum_{F\in \Sur(V, G)} \prod_{i=1}^{n+u}\bP(FA_i(n)=0),
\end{split}
\end{equation*}
where $A_i(n)$ denotes the $i$-th column of $A(n)$. More generally, for every $B\in\Hom(W,G)$ and $B_i=Bw_i \in G$, we have 
\begin{equation}\label{eq:column-factorization-sec2}
\bP(FA(n)=B)=\prod_{i=1}^{n+u}\bP(FA_i(n)=B_i).
\end{equation}

\begin{defn}\label{def2a}
Let $\omega>0$. We say that $F\in \Hom(V,G)$ is a code of distance $\omega$ if for every $\sigma \subset [n]$ with $|\sigma|<\omega$, we have $FV_{\backslash \sigma}=G$. Equivalently, $F$ remains surjective after deleting any set of fewer than $\omega$ standard basis vectors of $V$.
\end{defn}

To estimate probabilities for a code $F$, we first recall the following lemma that gives an estimate for the probability associated with a single column.

\begin{lem}\label{lem2b}
(\cite[Lemma 2.1]{Woo19}) Let $X\in R^n$ be a random vector with independent $\al$-balanced entries. Let $F\in\Hom(V,G)$ be a code of distance $\omega$. Then for every $B \in G$,
\begin{equation*}
\lt \bP(FX=B)-|G|^{-1} \rt \le \exp(-\al\omega/a^2).
\end{equation*}
\end{lem}

\begin{lem}\label{lem2c}
Fix $\delta>0$. If $F\in\Hom(V,G)$ is a code
of distance $\delta n$, then for all $B\in\Hom(W,G)$,
\begin{equation*}
\bP(FA(n)=B) = |G|^{-n-u}(1+o(1)).
\end{equation*}
The $o(1)$ term is uniform in the choice of the code $F$ and $B\in\Hom(W,G)$.
\end{lem}

\begin{proof}
Taking $\epsilon = \frac{\delta}{a^2} > 0$ in the condition \eqref{eq1_condition}, we have $\sum_{i=1}^{n+u}\exp(-\al_n(i)\delta n/a^2) \to 0$ as $n \to \infty$. Assume that $n$ is sufficiently large so that $|G|^{-1} > \exp(-\al_n(i)\delta n/a^2)$ for all $i$. By \eqref{eq:column-factorization-sec2} and Lemma \ref{lem2b}, we have
\begin{equation*}
\prod_{i=1}^{n+u}\left(|G|^{-1}-\exp(-\al_n(i)\delta n/a^2) \right)
\le \bP(FA(n)=B) 
\le \prod_{i=1}^{n+u}\left(|G|^{-1}+\exp(-\al_n(i)\delta n/a^2) \right).
\end{equation*}
After multiplying by $|G|^{n+u}$, this becomes
\begin{equation} \label{eq23a}
\prod_{i=1}^{n+u}\left(1-|G|\exp(-\al_n(i)\delta n/a^2) \right)
\le \frac{\bP(FA(n)=B)}{|G|^{-n-u}}
\le \prod_{i=1}^{n+u}\left(1+|G|\exp(-\al_n(i)\delta n/a^2) \right).
\end{equation}
The right-hand side of \eqref{eq23a} is bounded above by
\begin{equation*}
\prod_{i=1}^{n+u} \exp\left(|G|\exp(-\al_n(i)\delta n/a^2) \right) = \exp\left(|G|\sum_{i=1}^{n+u}\exp(-\al_n(i)\delta n/a^2) \right), 
\end{equation*}
which converges to $1$ as $n \to \infty$. On the other hand, setting $x_i=|G|e^{-\al_n(i)\delta n/a^2}>0$, we have
\begin{equation*}
\lt \prod_{i=1}^{n+u}(1-x_i)-1 \rt 
\le \prod_{i=1}^{n+u}(1+|x_i|)-1 
= \prod_{i=1}^{n+u}(1+x_i)-1
\end{equation*}
so the left-hand side of \eqref{eq23a} also converges to $1$ as $n \to \infty$.
\end{proof}

We next consider the case where $F$ is not a code. To control this case, we further divide the non-code case into finer subcases. We recall the notion of robustness and some lemmas from \cite{NW22}. For an integer $D$ with a prime factorization $\prod_i p_i^{e_i}$, we define $\ell(D):=\sum_i e_i$.

\begin{defn}\label{def2d}
(\cite[Definition 4.8]{NW22}) Given $\delta>0$, we say that $F\in \Hom(V,G)$ is \textit{$\delta$-robust} for a subgroup $H$ of $G$ if $H$ is minimal such that
\begin{equation*}
\#\{i\in[n] : Fv_i\notin H\}\le \ell([G:H])\delta n.
\end{equation*}
\end{defn}

\begin{lem} \label{lem2e}
(\cite[Lemma 4.10]{NW22}) Let $\delta>0$, $H$ be a subgroup of $G$ of index $D>1$ and $H=G_{\ell(D)}\subset\cdots\subset G_{2}\subset G_{1}\subset G_{0}=G$ be a maximal chain of proper subgroups. Let $p_{j}=|G_{j-1}/G_{j}|$. The number of maps $F\in\Hom(V,G)$ that are $\delta$-robust for $H$ and satisfy $\# \{ i : Fv_{i}\in G_{j-1}\backslash G_{j} \} = k_j$ for each $1\le j\le \ell(D)$ is at most
\begin{equation*}
|H|^{n-\sum_{j}k_{j}}\prod_{j=1}^{\ell(D)}\binom{n}{k_{j}}|G_{j-1}|^{k_{j}}.
\end{equation*}
\end{lem}

We have replaced $w_j$ in \cite{NW22} by $k_j$ to avoid confusion with the standard basis $w_1, \ldots, w_{n+u}$ of $W$.

\begin{lem} \label{lem2f}
(\cite[Lemma 4.11]{NW22}) Let $\delta$, $F$, $H$, $D$, $H=G_{\ell(D)}\subset \cdots\subset G_{2}\subset G_{1}\subset G_{0}=G$ and $p_j$ be as in Lemma \ref{lem2e}. (In particular, $F$ is $\delta$-robust.) For $1\le j\le \ell(D)$, let $k_{j}$ be the number of $i\in [n]$ such that $Fv_{i}\in G_{j-1}\backslash G_{j}$. Let $X \in R^{n}$ be an $\al$-balanced random vector. Then
\begin{equation*}
\bP(FX=0) \le \left( D|G|^{-1}+\exp(-\al\delta n/a^{2}) \right) \prod_{j=1}^{\ell(D)} \left( p_{j}^{-1}+\frac{p_{j}-1}{p_{j}}\exp(-\al k_{j}/p_{j}^{2}) \right).
\end{equation*}
\end{lem}

We now prove two additional lemmas that will be used in the proof of Theorem \ref{thm1f}. 

\begin{lem} \label{lem2g}
Let $C\ge 1$ and $f>0$. If $\delta>0$ is sufficiently small in terms of $C$ and $f$, then for all sufficiently large $n$,
\begin{equation*}
\sum_{m=1}^{\lf \delta n\rf}\binom{n}{m}C^m\le \exp(fn).
\end{equation*}
\end{lem}

\begin{proof}
Assume that $\delta < 1/2$. Then
\begin{equation*}
\sum_{m=1}^{\lf \delta n\rf}\binom{n}{m}C^m
\le (\lf \delta n\rf+1) \binom{n}{\lf \delta n\rf} C^{\lf \delta n\rf}
\le (\delta n+1)\left(\frac{e C}{\delta}\right)^{\delta n}.
\end{equation*}
Since $\delta\log(eC/\delta)\to 0$ as $\delta\to0$, we may choose
$\delta$ sufficiently small so that
\begin{equation*}
\log(\delta n+1)+\delta n\log(eC/\delta)\le fn
\end{equation*}
for all sufficiently large $n$. This proves the claim.
\end{proof}

\begin{lem}\label{lem2h}
Assume that the condition \eqref{eq1_condition} holds. Let $q$ be a prime and $C \ge 1$. If $\delta>0$ is sufficiently small in terms of $q$ and $C$, then
\begin{equation} \label{eq2_layer}
\lim_{n \to \infty} \sum_{m=1}^{\lf \delta n\rf} \binom{n}{m}C^m \prod_{i=1}^{n+u} \left(q^{-1}+\frac{q-1}{q}\exp(-\al_n(i)m/q^2)\right) = 0.
\end{equation}
\end{lem}

\begin{proof}
Let $\gamma=(q-1)/q$ and $x_i=\exp(-\al_n(i)/q^2)$ for $1\le i\le n+u$. By the condition \eqref{eq1_condition},
\begin{equation*}
\sum_{i=1}^{n+u}x_i^{\delta n} = \sum_{i=1}^{n+u}\exp(-\al_n(i)\delta n/q^2) \to 0
\end{equation*}
as $n \to \infty$, so $\sum_{i=1}^{n+u}x_i^{\delta n} \le 1$ when $n$ is sufficiently large. In that case, Jensen's inequality gives
\begin{equation*}
\frac{1}{n+u}\sum_{i=1}^{n+u}x_i^m \le \left(\frac{1}{n+u}\sum_{i=1}^{n+u}x_i^{\delta n}\right)^{\frac{m}{\delta n}}
\le (n+u)^{-\frac{m}{\delta n}}.
\end{equation*}
By the AM--GM inequality,
\begin{align*}
\prod_{i=1}^{n+u}\left(q^{-1}+\gamma x_i^m\right)
& \le \left(q^{-1}+\frac{\gamma}{n+u}\sum_{i=1}^{n+u}x_i^m\right)^{n+u} \\
& \le \left(1-\gamma\left(1-(n+u)^{-\frac{m}{\delta n}}\right)\right)^{n+u} \\
& \le \exp\left(-\gamma\left(1-n^{-\frac{m}{\delta n}}\right)(n+u)\right) \\
& \le \exp\left(-\gamma(1-e^{-t})n\right)
\end{align*}
where $t=\frac{m\log n}{\delta n}$. Now we split the sum into the ranges $t<1$ and $t \ge 1$. 

First assume that $t<1$. Then $1-e^{-t} \ge c_0 t$ for some absolute constant $c_0>0$, so
\begin{equation*}
\prod_{i=1}^{n+u}\left(q^{-1}+\gamma x_i^m\right)
\le \exp\left(-\gamma c_0 tn\right)
= n^{-\gamma c_0m/\delta}.
\end{equation*}
Assume that $\delta$ is sufficiently small so that $\gamma c_0 / \delta > 2$. The contribution from the terms with $t<1$ is at most
\begin{equation*}
\sum_{m=1}^{\lf \delta n\rf} \binom{n}{m}C^m n^{-\gamma c_0m/\delta}
\le \sum_{m=1}^{\infty} (Cn^{1-\gamma c_0/\delta})^m
\le \sum_{m=1}^{\infty} (Cn^{-1})^m = o(1).
\end{equation*}

Now assume that $t \ge 1$. Then 
\begin{equation*}
\prod_{i=1}^{n+u}\left(q^{-1}+\gamma x_i^m\right) 
\le \exp\left(-\gamma(1-e^{-1})n\right).
\end{equation*}
By Lemma \ref{lem2g}, after reducing $\delta$ further if necessary, we have
\begin{equation*}
\sum_{m=1}^{\lf\delta n\rf}\binom{n}{m}C^m\le \exp(fn)
\end{equation*}
for some $f<\gamma(1-e^{-1})$. Now the contribution from the terms with $t \ge 1$ is at most
\begin{equation*}
\exp(fn)\exp\left(-\gamma(1-e^{-1})n\right) = o(1). \qedhere
\end{equation*}
\end{proof}

\begin{proof}[Proof of Theorem \ref{thm1f}]
The case $G=0$ is trivial, so we may assume that $G \ne 0$. Let $a=p^d$ be the exponent of $G$, and let $\overline{A}(n)$ denote the reduction of $A(n)$ modulo $a$. Since every homomorphism to $G$ factors through reduction modulo $a$, we have
\begin{equation*}
\bE(\# \Sur(\cok(A(n)),G))
=\bE(\# \Sur(\cok(\overline{A}(n)),G))
=\sum_{F\in\Sur(V,G)}\bP(F\overline{A}(n)=0).
\end{equation*}
We note that the reduction modulo $a$ of an $\al_n(i)$-balanced random vector in $\zp^n$ is still $\al_n(i)$-balanced as a random vector in $R^n$.
In the remainder of the proof, we denote $\overline{A}(n)$ by $A(n)$ for simplicity. Let $\delta>0$ be a sufficiently small constant. Since
\begin{equation*}
|G|^{-u}=\sum_{F\in\Hom(V,G)}|G|^{-n-u},
\end{equation*}
we have
\begin{align*}
\left|\bE\bigl(\#\Sur(\cok(A(n)),G)\bigr)-|G|^{-u}\right| 
\le & \sum_{\substack{F\in\Sur(V,G)\\ F\text{ code of distance }\delta n}} \left|\bP(FA(n)=0)-|G|^{-n-u}\right| \\
+ & \sum_{\substack{F\in\Sur(V,G)\\ F\text{ not a code of distance }\delta n}} \bP(FA(n)=0) \\
+ & \sum_{\substack{F\in\Hom(V,G)\\ F\text{ not a code of distance }\delta n}} |G|^{-n-u}.
\end{align*}
The first term converges to $0$ as $n \to \infty$ by Lemma
\ref{lem2c}. Indeed, the $o(1)$ term in Lemma \ref{lem2c} is
independent of the code $F$, and there are at most $|G|^n$ choices for $F$. By the proof of \cite[Theorem 2.9]{Woo19}, if $\delta$ is sufficiently small, then the last term also converges to $0$ as $n \to \infty$. We fix such a $\delta$ and shrink it further below if necessary. It remains to bound the second term.

Since $G$ itself satisfies the defining inequality for $\delta$-robustness, every map $F\in\Hom(V,G)$ is $\delta$-robust for at least one subgroup of $G$. Moreover, if $F$ is not a code of distance $\delta n$, then there exists $\sigma\subset[n]$ with $|\sigma|<\delta n$ such that $FV_{\backslash\sigma}$ is a proper subgroup of $G$ and
\begin{equation*}
\#\{i\in[n]:Fv_i\notin FV_{\backslash\sigma}\}
\le |\sigma|<\delta n
\le \ell([G:FV_{\backslash\sigma}])\delta n.
\end{equation*}
Hence if $F$ is not a code of distance $\delta n$ and $\delta$-robust for $H$, then $H \ne G$. 
For each proper subgroup $H$ of $G$, fix a maximal chain
\begin{equation*}
H=G_{\ell(D)}\subset\cdots\subset G_2\subset G_1\subset G_0=G
\end{equation*}
($D=[G:H]$) and write $p_j=|G_{j-1}/G_j|$ for $1\le j\le \ell(D)$. If $F$ is $\delta$-robust for $H$, let 
\begin{equation*}
k_j=\#\{i\in[n]:Fv_i\in G_{j-1}\backslash G_j\}\quad(1\le j\le \ell(D)).
\end{equation*}
By the definition of $\delta$-robustness, we have $\sum_{j} k_j = \#\{ i \in [n] : Fv_i \notin H \} \le \ell(D)\delta n$.

Since there are only finitely many proper subgroups $H$ of $G$ and finitely many indices $j$ for each of them, we may assume that $\delta>0$ is sufficiently small so that Lemma \ref{lem2h} holds for $q=p_j$, $C=|G_{j-1}|/|H|$, and $\delta$ replaced by $\ell(D)\delta$, for every proper subgroup $H$ of $G$ and every $1\le j\le \ell(D)$ in the fixed chain above.

Fix a proper subgroup $H$ of $G$ and use the chain above. Let
\begin{equation*}
C_j=\frac{|G_{j-1}|}{|H|}\quad\text{and}\quad
\gamma_j=\frac{p_j-1}{p_j}\quad(1\le j\le \ell(D)).
\end{equation*}
By applying Lemma \ref{lem2f} to the $i$-th column $A_i(n)$ of $A(n)$, we have
\begin{equation*}
\bP(FA_i(n)=0)\le
\left(|H|^{-1}+\exp(-\al_n(i)\delta n/a^2)\right)
\prod_{j=1}^{\ell(D)}
\left(p_j^{-1}+\gamma_j\exp(-\al_n(i)k_j/p_j^2)\right).
\end{equation*}
Combining this estimate with Lemma \ref{lem2e} and \eqref{eq:column-factorization-sec2}, the contribution of surjective maps $F$ that are $\delta$-robust for this fixed $H$ is at most
\begin{align*}
&|H|^n\prod_{i=1}^{n+u}
\left(|H|^{-1}+\exp(-\al_n(i)\delta n/a^2)\right) \\
\times \, & \sum_{\substack{0\le k_j\le \lf \ell(D)\delta n\rf\\ \text{not all }k_j\text{ are }0}} \prod_{j=1}^{\ell(D)}
\left( \binom{n}{k_j}C_j^{k_j} \prod_{i=1}^{n+u} \left(p_j^{-1}+\gamma_j\exp(-\al_n(i)k_j/p_j^2)\right)
\right).
\end{align*}
Here, the surjectivity of $F$ actually gives the stronger condition $k_1\ge 1$, since otherwise $F(V)\subseteq G_1$. We have replaced this by the weaker condition that not all $k_j$ are zero, which is sufficient for the desired upper bound.

In the above formula, the first factor converges to $|H|^{-u}$ as in the proof of Lemma \ref{lem2c}. For each $j$, the $k_j=0$ term in the corresponding sum is $1$, while the sum over $1\le k_j\le\lf\ell(D)\delta n\rf$ is $o(1)$ by our choice of $\delta$ and Lemma \ref{lem2h}. Hence the second factor is
\begin{equation*}
\prod_{j=1}^{\ell(D)} \left( \sum_{k_j=0}^{\lf \ell(D)\delta n\rf} \binom{n}{k_j}C_j^{k_j} \prod_{i=1}^{n+u} \left(p_j^{-1}+\gamma_j\exp(-\al_n(i)k_j/p_j^2)\right) \right) - 1 = \prod_{j=1}^{\ell(D)} (1+o(1)) - 1 = o(1).
\end{equation*}
Thus the contribution from maps that are $\delta$-robust for the fixed subgroup $H$ is $o(1)$. Summing over the finitely many proper subgroups $H \le G$ proves that the second term is $o(1)$. Hence
\begin{equation*}
\lim_{n\to\infty}\bE(\#\Sur(\cok(A(n)),G))=|G|^{-u}. \qedhere
\end{equation*}
\end{proof}

\begin{exmp} \label{exmp2i}
Let $(t_n)_{n \ge 2}$ be a sequence of positive integers such that
$1 \le t_n \le n+u$ and $t_n=n^{o(1)}$. Let $(\al_n)_{n \ge 2}$
be a sequence in $(0,1/2]$ such that, for every constant $\Delta>0$, $\al_n \ge \frac{\Delta \log n}{n}$ for all sufficiently large $n$. Set $\al_n(i)=\frac{\log n}{n}$
for $1 \le i \le t_n$ and $\al_n(i)=\al_n$ for
$t_n<i \le n+u$. For each $n \ge 2$, let $A(n)$ be a random matrix
in $\M_{n \times (n+u)}(\zp)$ with independent entries whose $i$-th column is $\al_n(i)$-balanced.

Fix $\epsilon>0$. Assume that $n$ is sufficiently large so that $t_n \le n^{\epsilon/2}$ and $\al_n \ge \frac{2 \epsilon^{-1} \log n}{n}$. Then
\begin{align*}
\sum_{i=1}^{n+u} \exp \left( - \epsilon \al_n(i) n \right) 
& = t_n \exp(-\epsilon \log n) + (n+u - t_n) \exp \left(-\epsilon \al_n n \right) \\
& \le \frac{n^{\epsilon/2}}{n^{\epsilon}} + (n+u) \exp \left( -\epsilon \frac{2 \epsilon^{-1} \log n}{n} n \right) \\
& \le \frac{1}{n^{\epsilon/2}} + \frac{n+u}{n^2}
\end{align*}
so
\begin{equation*}
\lim_{n \to \infty} \sum_{i=1}^{n+u} \exp \left( - \epsilon \al_n(i) n \right) =0.
\end{equation*}
By Theorem \ref{thm1c}, the cokernels $\cok(A(n))$ converge in distribution to the limiting distribution of Haar-random matrices. This example is not covered by the previous homogeneous results, such as Theorem \ref{thm1a} or its improvement \cite[Theorem 1.5(a)]{JLY26}.

This contrasts with the homogeneous critical case. For example, let $\zeta_n$ be a random element of $\zp$ satisfying $\bP(\zeta_n=0) = 1 - \frac{\log n}{n}$ and $\bP(\zeta_n=1) = \frac{\log n}{n}$, and let $A(n)$ be a random $n \times (n+u)$ matrix over $\zp$ whose entries are i.i.d. copies of $\zeta_n$. Then a fixed row of $A(n)$ is zero with probability $\left(1-\frac{\log n}{n}\right)^{n+u}$, so the probability that $A(n)$ has a zero row is
\begin{equation*}
1 - \left( 1 - \left(1-\frac{\log n}{n}\right)^{n+u} \right)^n,
\end{equation*}
which converges to $1-e^{-1}$ as $n \to \infty$ (cf. \cite[p. 4496]{Woo22}). Hence $\cok(A(n))$ cannot converge to the limiting distribution of Haar-random matrices. Thus Theorem \ref{thm1c} covers broader classes of random matrices than existing results under homogeneous balancedness conditions.
\end{exmp}

\section*{Acknowledgments}

Jungin Lee was supported by the National Research Foundation of Korea (NRF) grant funded by the Korea government (MSIT) (No. RS-2024-00334558 and No. RS-2025-02262988). 
Sungjin Park was supported by the National Research Foundation of Korea (NRF) grant funded by the Korea government (MSIT) (No. RS-2024-00334558).


\end{document}